\documentclass[secthm,amsthm]{amsart}
\usepackage{comment}

\usepackage[english]{babel}
\usepackage{amsmath, amssymb}
\usepackage{hyperref}

\newcommand{\iaoi}{if and only if }

\newcommand{\mrm}{\mathrm}
\newcommand{\mcl}{\mathcal}

\newcommand{\R}{\text{$\mathbb{R}$}}
\newcommand{\N}{\text{$\mathbb{N}$}}

\newcommand{\FF}{\text{$\mathcal{F}$}}

\newcommand{\eps}{\text{$\varepsilon$}}
\newcommand{\ph}{\text{$\varphi$}}

\DeclareMathOperator*{\flim}{\mathcal{F}-lim}
\DeclareMathOperator*{\dflim}{(2\mathcal{F})-lim}
\DeclareMathOperator*{\limstat}{\operatorname{limstat}}

\newcommand{\abs}[1]{\left\lvert#1\right\rvert}

\newcommand{\llip}[1]{\left\lfloor #1 \right\rfloor}

\newcommand{\intrv}[2]{[#1,#2]}

\newcounter{low}
\renewcommand{\thelow}{\rm (\alph{low})}
\newenvironment{lista}{\begin{list}{\thelow}{\usecounter{low}\setlength{\leftmargin}{3pc}\setlength{\itemsep}{0pc}}\setlength{\labelwidth}{10pc}}{\end{list}}

\newcommand{\refeq}[1]{{\rm (\ref{#1})}}
\numberwithin{equation}{section}

\newcommand{\inv}[1]{\text{${#1}^{-1}$}}
\newcommand{\dd}{\mrm{d}}
\newcommand{\limti}[1]{\lim\limits_{#1\to\infty}}

\newcommand{\ud}{\text{$\overline{d}$}}

\newcommand{\ld}{\text{$\underline{d}$}}

\theoremstyle{definition}
\newtheorem{definition}{Definition}[section]
\newtheorem{example}[definition]{Example}

\theoremstyle{plain}
\newtheorem{theorem}[definition]{Theorem}
\newtheorem{lemma}[definition]{Lemma}
\newtheorem{proposition}[definition]{Proposition}

\theoremstyle{remark}

\begin{document}

\title{L\'evy group and density measures}
\author{Martin Sleziak}
\address{Department of Algebra, Geometry and Mathematical Education, Faculty of Mathematics, Physics and Informatics, Comenius University, Mlynsk\'a dolina, 842 48 Bratislava, Slovakia}
\email{\tt sleziak@fmph.uniba.sk, ziman@fmph.uniba.sk}
\thanks{The first author was supported by VEGA Grant 1/3020/06 and by
Comenius University Grant UK/398/2007}
\author{Milo\v s Ziman}
\thanks{The second author was supported by
Comenius University Grant UK/336/2005} 

\begin{abstract}
We will deal with finitely additive measures on integers extending
the asymptotic density. We will study their relation to the L\'evy
group $\mathcal{G}$ of permutations of $\mathbb N$. Using a new
characterization of the L\'evy group $\mathcal G$ we will prove
that a finitely additive measure extends density if and only if it
is $\mathcal{G}$-invariant.
\end{abstract}

\maketitle

\section*{Introduction}

L\'evy group $\mcl G$ is a group of permutations of positive
integers which is tightly linked to the notion of asymptotic
density. The connection between this group and asymptotic density
(as well as several related notions) was studied e.g.~by
Bl\"umlinger \cite{blumlinger}, Obata \cite{obata, blumoba}. Some
other groups related to (extensions of) asymptotic density were
also studied, we can mention recent papers of Nathanson and Parikh
\cite{nathpari} or Giuliano Antonini and Pa\v{s}t\'eka
\cite{GIULPAST}.

In this paper we will study the connection between the L{\'e}vy
group and finitely additive measures on integers extending the
asymptotic density. We will call such measures density measures.
The term density measures was probably coined by Dorothy Maharam
\cite{maharam}. They were studied (among many others) by Blass,
Frankiewicz, Plebanek and Ryll--Nar\-dzew\-ski in \cite{bfpr}, van
Douwen in \cite{vandouwen} or \v Sal\'at and Tijdeman in
\cite{st}.

Both the L\'evy group and the density measures have found
applications in number theory and, more recently, in the theory of
social choice (see e.g.~Fey \cite{fey}, Lauwers \cite{lauwers}).

The main purpose of this paper is to show that the density
measures are precisely the finitely additive measures which are
$\mcl G$-invariant. The $\mcl G$-invariant measures were studied
by Bl\"umlinger in \cite{blumlinger}. Bl\"umlinger and Obata deal
with the $\mcl G$-invariant extensions of Ces\'aro mean in
\cite{blumoba}.

We also obtain an interesting characterization of the L\'evy group
in terms of statistical convergence.

\section{Preliminaries}

We start by defining the two central notions of this paper -- the
L\'evy group and the density measures -- and mentioning a few
necessary facts about them.

\begin{definition}
The \emph{asymptotic density} of a set $A\subseteq\N$ is defined
by $d(A) = \lim\limits_{n \to \infty} \frac{A(n)}{n}$, where $A(n)
= \big|A \cap [1,n]\big|$.
We denote the collection of sets having asymptotic density by $\mcl{D}$.

A \emph{density measure} is a finitely additive measure on $\N$
which extends the asymptotic density; i.e., it is a function
$\mu:\, \mcl{P}(\N) \to \intrv 01$ satisfying the following
conditions:
\begin{lista}
\item\label{dm0} $\mu(\N) = 1$;
\item\label{dm1} $\mu(A \cup B) =
\mu(A) + \mu(B)$ for all disjoint $A, B \subseteq \N$;
\item\label{dm2} $\mu|_{\mcl{D}} = d$.
\end{lista}
\end{definition}
(Throughout the paper a \emph{measure} will mean a set function on
$\mcl P(\N)$ fulfilling \ref{dm0} and \ref{dm1}.)

Density measures can be constructed using a limit along an
ultrafilter. The set of all free ultrafilters on $\N$ will be
denoted by $\beta\N^*$. For $\FF\in\beta\N^*$ and a bounded sequence
$(x_n)$ we denote by $\flim x_n$ the limit of this sequence along
the ultrafilter $\FF$ (see \cite[p.122, Definition 8.23]{balste},
\cite[p.206, Definition 2.7]{hrjech} for definition and basic
properties of a limit along an ultrafilter).

For any $\FF\in\beta\N^*$ the function
$$\mu_{\FF}(A)=\flim \frac{A(n)}n$$
is a density measure (see e.g.~\cite[Theorem 8.33]{balste},
\cite[p.207]{hrjech}). We will use this construction of a density
measure several times. Another possibility to show the
existence of density measures relies on Hahn-Banach theorem.

\begin{definition}
The \emph{L\'evy group} $\mcl G$ is the group of all permutations
$\pi$ of $\N$ satisfying
\begin{equation}\label{levy}
    \lim_{n \to \infty}\frac{\big|\{ k;\, k \leq n < \pi(k)\}\big|}{n} = 0.
\end{equation}
\end{definition}

We will need the following characterization of the L\'evy group.
\begin{lemma}{\cite[Lemma 2]{blumlinger}} \label{lmBlum}
A permutation $\pi$ of $\mathbb N$ belongs to
$\mcl{G}$ \iaoi
\begin{gather}\label{blum eq}
    \lim_{n\to\infty} \frac{A(n)-(\pi A)(n)}{n} = 0
\end{gather}
for each $A\subseteq \N$.
\end{lemma}

For more information about the L\'evy group see \cite{blumlinger}
and \cite{obata}.

\section{$\mcl{G}$-invariance}

To answer the question of $\mcl{G}$-invariance of a density
measure we use the representation of the L\'evy group $\mcl{G}$
with the help of statistical convergence.

We say that a real sequence $(x_n)$ {\em converges statistically
to $L$} ($\limstat\limits_{n \to \infty} x_n = L$) if for every
$\varepsilon>0$ the set
\begin{gather*}
A_\varepsilon = \{n;\,\abs{x_n-L} \geq \varepsilon \}
\end{gather*}
has zero density ($d(A_\varepsilon)=0$).

The following result is well-known (see Fridy \cite[Theorem
1]{fridy} or \v{S}al\'at \cite[Lemma 1.1]{sal1980}).

\begin{theorem}\label{fridy thm}
A sequence $(x_n)$ is statistically convergent to $L \in \R$ if
and only if there exists a set $A$ such that $d(A)=1$ and the
sequence $(x_n)$ converges to $L$ along the set $A$, i.e., $L$ is
limit of the subsequence $(x_n)_{n\in A}$.
\end{theorem}

\begin{theorem}\label{levy stat}
A permutation $\pi: \N \to \N$ belongs to $\mcl{G}$ \iaoi
\begin{gather}\label{levy stat eq}
    \limstat_{n \to \infty} \frac{\pi(n)}n=1.
\end{gather}
\end{theorem}

\begin{proof}
Let $\pi\in\mcl G$. Suppose we are given $\eps>0$. Denote
\begin{gather*}
A=\{k;\, \pi(k) - k > \eps k\} = \{k;\, \pi(k) > (1+\eps)k\},\\
B=\{k;\, k - \pi(k) > \eps k\} = \{k;\, \pi(k) < (1-\eps)k\}.
\end{gather*}
Obviously $C=A\cup B=\{k;\,\abs{\frac{\pi(k)}k-1}>\eps\}$. So it suffices to
show that $d(C)=0$, i.e., $\lim\limits_{n\to\infty} C(n)/n = 0$.

If $\pi(k)\leq n$ for some $k\in A$, we get $(1+\eps)k\leq \pi(k)
\leq n$ and $k\leq \frac n{1+\eps}$. Hence $(\pi A)(n)\leq
A\left(\llip{\frac n{1+\eps}}\right)$ and Lemma \ref{lmBlum} 
yields
\begin{gather*}
\limsup_{n \to \infty} \frac{A(n)}n =
\limsup_{n \to \infty} \frac{\pi A(n)}n \leq
\limsup_{n \to \infty} \frac{A\left(\llip{\frac n{1+\eps}}\right)}n\\
= \limsup_{n \to \infty} \frac{A\left(\llip{\frac n{1+\eps}}\right)}{\llip{\frac n{1+\eps}}}\frac{\llip{\frac n{1+\eps}}}n \leq
\limsup_{n \to \infty} \frac{A(n)}n \frac1{1+\eps}.
\end{gather*}
This implies immediately $d(A)=\limsup\limits_{n \to \infty}
\frac{A(n)}n =0$.

To show that $d(B)=0$ we can proceed analogously. Another
possibility is to notice that $B=\{k; \pi(k)<(1-\eps)k\}=
\inv\pi(\{l; l<(1-\eps)\inv\pi(l)\})\subseteq \inv \pi(\{l;
(1+\eps)l < \inv\pi(l)\})$ and repeat the same argument for the
permutation $\inv\pi\in\mcl{G}$.

Thus we get $d(C)=d(A)+d(B)=0$.

To prove the reverse implication assume that $\pi$ satisfies \refeq{levy stat eq}.
As before, taking $\eps >0$ let us denote $A =\{k;\, \pi(k) > (1+\eps)k\}$. Then $d(A)=0$.

If $k\leq n <\pi(k)$, then either $k\in A$ or $k>\frac n{1+\eps}$ (otherwise
$\pi(k) \leq (1+\eps) k \leq n$, contradicting $n < \pi(k)$). So
\begin{gather*}
\abs{\{ k;\, k\leq n <\pi(k) \}} \leq A(n) + n\left(1-\frac 1{1+\eps}\right) + 1\leq A(n) + n\eps + 1,\\
\limsup_{n \to \infty} \frac{\abs{\{ k;\, k\leq n <\pi(k) \}}}n \leq \eps + \lim\limits_{n \to \infty}
\frac{A(n)}n =\eps.
\end{gather*}
Since $\eps$ can be chosen arbitrarily small, we get
\begin{gather*}
\lim_{n \to \infty} \frac{\abs{\{ k;\, k\leq n <\pi(k) \}}}n = 0,
\end{gather*}
and $\pi \in \mcl{G}$.
\end{proof}

Van Douwen (see \cite[Theorem 1.12]{vandouwen})  characterized
density measures using invariance with respect to a~particular
kind of permutations.
\begin{theorem}\label{van douwen char}
A measure $\mu$ on $\N$ is a density measure if and only if
 $\mu(A) = \mu(\pi A)$ for all $A \subseteq \N$ and
all permutations $\pi: \N \to \N$ such that
\begin{gather}\label{van douwen eq}
\lim\limits_{n \to \infty} \frac{\pi (n)}{n} = 1.
\end{gather}
\end{theorem}

Van Douwen proved even more, but here we only need the above
result.

One can see easily that if a permutation $\pi$ fulfills \refeq{van
douwen eq}, then it fulfills also \refeq{levy stat eq}. Using this
fact we get

\begin{proposition}
If a measure $\mu$ on $\N$ is $\mcl{G}$-invariant, then it is a density measure.
\end{proposition}

This result can also be deduced from Bl\"umlinger and Obata
\cite[Theorem 2]{blumoba}, where it was proved by different means.

Next we will show that the reverse of this proposition is true as well.

\begin{proposition}\label{prop Ginv}
If $\pi\in\mcl{G}$ and $\mu$ is a density measure, then for each
$A\subseteq\N$,
\begin{equation}\label{G invariance eq}
    \mu(\pi A)=\mu(A).
\end{equation}
\end{proposition}

\begin{proof}
Let $\pi\in\mcl{G}$, $A \subseteq \N$ and $B=\pi A$. Define
\begin{gather*}
A' = A\smallsetminus(A\cap B) = \{a_1 < a_2 < a_3 < \ldots\},\\
B' = B\smallsetminus(A\cap B) = \{b_1 < b_2 < b_3 < \ldots\}.
\end{gather*}

As $A = A' \cup (A\cap B)$ and $B = B' \cup (A\cap B)$, it suffices to prove that $\mu(A') = \mu(B')$.

Without loss of generality we can assume that $A'$ and $B'$ are infinite,
otherwise we get $d(A')=d(B')=0$ (if
one of the sets is finite, then it has zero density and since $\lim\limits_{n \to \infty}
\frac{A'(n)-B'(n)}{n}= \lim\limits_{n \to \infty}\frac{A(n)-B(n)}{n} = 0$,
the other one has the same density) and $\mu(A)=\mu(A\cap
B)=\mu(B)$.

Let us define a permutation $\varphi : \N \to \N$ by
\begin{gather*}
\begin{array}{cl}
\varphi(n)=n &\text{if }n\notin A'\cup B';\\
\varphi(a_i)=b_i &\text{for }i=1,2,\ldots;\\
\varphi(b_i)=a_i &\text{for }i=1,2,\ldots
\end{array}
\end{gather*}
We claim that $\varphi\in\mcl{G}$. Indeed, as one of the sets
$\{i;\, a_i\leq n < b_i\}$, $\{i;\, b_i\leq n < a_i\}$
is empty, we get  $\abs{\{k;\, k\leq n < \varphi(k)\}}=
\abs{\{i;\, a_i\leq n < b_i\}} + \abs{\{i;\, b_i\leq n < a_i\}} =
\abs{A'(n)-B'(n)}=\abs{A(n)-B(n)}$ and $\limti n \frac{\abs{A(n)-B(n)}}{n} = 0$ by Lemma \ref{lmBlum}. 
Moreover $\varphi(A')=B'$ and $\inv\varphi=\varphi$.

By 
Theorems \ref{fridy thm} and \ref{levy stat}
there exists a set $F$ such that $d(F)=0$ and
\begin{gather}\label{lim - F}
\lim\limits_{\substack{n\in\N\smallsetminus F\\n\to\infty}} \frac{\varphi(n)}{n} = 1.
\end{gather}
Set $F':=A'\cap (F\cup
\varphi F)$ and $E:=F'\cup \varphi F'$. Clearly $F'\subseteq A'$
and $\varphi F'=B'\cap(F\cup \varphi F) \subseteq B'$. Since the
permutations in $\mcl{G}$ preserve density, we get $d(\varphi
F)=0$ and $d(F\cup\varphi F) = 0$. Thus $d(F')=d(\varphi F')=0$.

We modify the permutation $\varphi$ a little bit to get a permutation satisfying
\refeq{van douwen eq}.
$$\psi(n)=
\begin{cases}
n, &n\in E;\\
\varphi(n), &n\notin E.
\end{cases}$$
(The equality $\varphi E=E$ holds since $\varphi = \inv\varphi$.
So by changing the permutation $\varphi$ on the set $E$ to
identity map we get again a permutation of $\N$.) If $n\in F$,
then either $n\in E$ and $\psi(n)=n$, or $n\notin A'\cup B'$ and
$\psi(n)=\varphi(n)=n$. Hence, using \refeq{lim - F} we get
\begin{gather*}
\lim\limits_{n\to\infty} \frac{\psi(n)}{n}=1.
\end{gather*}
Moreover $(A'\smallsetminus F')\cap E = (A'\smallsetminus F')\cap \varphi
F' \subseteq A' \cap B' = \emptyset$. Thus $A'\smallsetminus F'$ and $E$ are
disjoint and $\psi$ coincides with $\varphi$ on the set $A'\smallsetminus F'$.
Then $\psi(A'\smallsetminus F') =\varphi(A'\smallsetminus F') =B'\smallsetminus \varphi F'$.

Now, by Theorem \ref{van douwen char}
we get $\mu(A'\smallsetminus F')=\mu(\psi(A'\smallsetminus F'))=\mu (B'\smallsetminus \varphi F')$,
and finally
\begin{gather*}
\begin{aligned}
\mu(A') &= \mu(A'\smallsetminus F')+\mu(F')=\mu(A'\smallsetminus F')\\
&= \mu(B'\smallsetminus\varphi F') = \mu(B'\smallsetminus\varphi F') + \mu(\varphi F') =\mu(B').
\end{aligned}
\end{gather*}
\end{proof}

The last two propositions together give us the main result of this
paper.

\begin{theorem}\label{G invariance}
A measure $\mu$ on $\N$ is a density measure if and only if
it is $\mcl{G}$-in\-va\-riant, i.e.,  $\mu(A) = \mu(\pi A)$ for all $A \subseteq \N$ and
all permutations $\pi \in \mcl{G}$.
\end{theorem}

\section{Applications}

We proved in Proposition \ref{prop Ginv} that every density
measure is $\pi$-invariant for permutations $\pi\in\mcl G$. It is
natural to ask whether there are other permutations with this
property. Proposition \ref{prop piinv} states that this property
characterizes L\'evy group.

\begin{proposition}\label{prop piinv}
If $\pi$ is a permutation such that every density measure is
$\pi$-invariant, i.e., $\mu(\pi A)=\pi A$ for every $A\subset\N$
and every density measure $\mu$, then $\pi\in\mcl G$.
\end{proposition}

\begin{proof}
Suppose that there is a permutation $\pi\notin\mcl{G}$ such that
every density measure $\mu$ is $\pi$-invariant. By Lemma \ref{lmBlum} 
then there exist a set $A\subseteq\N$ and a sequence $n_k$ with
$$\limti k \frac{A(n_k)-(\pi A)(n_k)}{n_k}=a>0.$$
Then any free ultrafilter $\FF$ with $\{n_k; k\in\N\}\in\FF$
yields a density measure $\mu(A)=\flim\frac{A(n)}n$ such that
$\mu(A)=\mu(\pi A)+a$.
\end{proof}

Using some known facts on the L\'evy group we can characterize the
pairs of sets having the property $\mu(A)=\mu(B)$ for every
density measure $\mu$. We will need the following result.

\begin{lemma}{\cite[Lemma 3]{blumlinger}} \label{lm Blum}
Let  $A,B\subseteq \N$ such that $A,B,\N\smallsetminus A,\N\smallsetminus B$ are infinite sets.
Then there is a $\pi\in\mcl G$ with $B=\pi A$ \iaoi $\limti n \frac{A(n)-B(n)}n=0$.
\end{lemma}

\begin{proposition}
Let $A,B\subseteq\N$. Then $\limti n \frac{A(n)-B(n)}n=0$ \iaoi $\mu(A)=\mu(B)$ for every
density measure $\mu$.
\end{proposition}

\begin{proof}
Assume that $\limti n \frac{A(n)-B(n)}n=0$. If we moreover assume that $A$ and
$B$ are neither finite nor cofinite, then by Lemma \ref{lm Blum} there exists a permutation
$\pi\in\mcl G$ with $B=\pi A$ and therefore $\mu(B)=\mu(A)$ by Proposition \ref{prop Ginv}. If
one of the sets $A$, $B$ is finite, then $d(A)=d(B)=0$. If one of them is cofinite, then $d(A)=d(B)=1$, in both
cases $\mu(A)=\mu(B)$.

On the other hand if $\mu(B)=\mu(A)$ holds for every density
measure $\mu$, then, in particular, for every $\mcl F\in\beta\N^*$
we get $\flim \frac{A(n)}n = \flim \frac{B(n)}n$.

This implies that $\flim \frac{A(n)-B(n)}n = 0$ for every $\mcl F\in\beta\N^*$. Thus the only
limit point of the sequence $(\frac{A(n)-B(n)}n)$ is 0 and $\limti n \frac{A(n)-B(n)}n=0$.
\end{proof}

The above proposition yields an alternative proof of Proposition
\ref{prop piinv}: Assume that the equality
$\mu(\pi A)=\mu(A)$ holds
for every density measure $\mu$.
By the above proposition we get $\limti
n \frac{(\pi A)(n)-A(n)}n= 0$ and $\pi\in\mcl G$.

The rest of this section will be devoted to a counterexample
answering some questions concerning density measures posed in
\cite{st} and \cite{vandouwen}. This example can be found in
Bl\"umlinger's paper \cite{blumlinger}. \footnote{Let us note that
the authors originally suggested an example that was much more
complicated. The possibility of using the following example was
pointed out by a referee.}

Let us recall that for any $A\subseteq\N$ the upper asymptotic
density is given by
$\ud(A)=\limsup\limits_{n\to\infty}\frac{A(n)}n$ and the lower
asymptotic density is given by
$\ld(A)=\liminf\limits_{n\to\infty}\frac{A(n)}n$.

\begin{example}\label{exblum}
Let $\FF$ be any free ultrafilter on $\mathbb{N}$.
By $2\FF$ we denote the ultrafilter
given by the base $\{2A; A\in\FF\}$, i.e., $2\FF=\{B\subseteq\N;
B\supseteq 2A$ for some $A\in\FF\}$. Let us define $\mu$ by
$$\mu(A)=2\dflim \frac{A(n)}n - \flim \frac{A(n)}n.$$

This function is shown to be a $\mcl{G}$-invariant measure in
\cite[p.5092--5093]{blumlinger}, hence by Theorem \ref{G
invariance} it is a density measure. For the sake of completeness
we will sketch the proof of this fact.

The estimates $\frac12\frac{A(n)}n \leq \frac{A(2n)}{2n} \leq
\frac12+\frac12\frac{A(n)}n$ imply $$\frac12 \flim \frac{A(n)}n
\leq \dflim \frac{A(n)}n \leq \frac12 + \frac12 \flim
\frac{A(n)}n.$$ From this we obtain $\mu(A)\in \intrv01$.

It is clear that $\mu(A)=d(A)$ whenever $A$ has density. Finite
additivity of $\mu$ follows from the additivity of $\FF$-limit.
Hence $\mu$ is indeed a density measure for any free ultrafilter $\FF$.

Now let us consider the set $A=\bigcup\limits_{i=1}^\infty
\{2^{2^i},2^{2^i}+1,\ldots,2.2^{2^i}-1\}$. Note that
$A(2.2^{2^i}-1)\geq\frac12$ and $A(2^{2^i}-1)\leq \frac1{2^{i-3}}$
for any positive integer $i$. It can be shown that $\ud
(A)=\frac12$ and $\mu(A)=1$ for any free ultrafilter containing the set
$\{2^{2^i}; i\in\N\}$.
\end{example}

Van Douwen asked in \cite[Question 7A.1]{vandouwen} whether
$\mu(A)\leq \ud(A)$ for every density measure. The same question
was asked again in the survey \cite{OPENDOUW}. The measure $\mu$
and the set $A$ from the above example answer this question in
negative.

This also yields a counterexample to the following claim of
Lauwers
\cite[p.46]{lauwers}:\\
{\em Every density measure can be expressed in the form
\begin{gather}\label{lauwers eq}
\mu_\ph(A) = \int_{\beta\N^*}\flim \frac{A(n)}{n}\, \dd\ph(\FF),
\qquad A \subseteq \N
\end{gather}
for some probability Borel measure $\ph$ on the set of all free
ultrafilters $\beta\N^*$.}\\
It is easy to note that if this claim were true the answer to van
Douwen's question would be positive. (The Lauwers' claim was
falsified already by Bl\"umlinger \cite{blumlinger}. In this paper
the set of all measures expressible by \refeq{lauwers eq} is
denoted by $\mcl R$ and it is shown to be a proper subset of the
set of all $\mcl G$-invariant measures.)

Let us note that Lauwers has shown in \cite[Lemma 4]{lauwers} that
a permutation $\pi$ preserves density measures of the form
\refeq{lauwers eq} \iaoi $\pi\in\mcl{G}$. The proof is similar to
our proof of Proposition \ref{prop piinv}. (The permutations
preserving density measures of the form \refeq{lauwers eq} are
called bounded in \cite{lauwers}.)

\v{S}al\'at and Tijdeman have posed another question concerning
the density measures \cite[p.201]{st}. They ask whether every
density measure has the following properties:\\
(a) If $A(n)\leq B(n)$ for all $n\in\N$ then $\mu(A)\leq\mu(B)$ (where $A,B\subseteq\N$).\\
(b) If $\limti n \frac{A(n)}{B(tn)}=1$ then $\mu(A)=t\mu(B)$ (where
$A,B\subseteq\N$ and $t\in\R$).\\
(The authors of \cite{st} conjectured that there exist density
measures that do not fulfill (a) and (b). We will see that this
conjecture was right.)

Clearly, any density measure of the form \refeq{lauwers eq} has
both these properties.

It is easy to verify that for the set $A$ from the preceding
example (and the measure given by an ultrafilter containing
$\{2^{2^i}; i\in\N\}$) we get $\mu(2A)=0$ and $\mu(A)=1$. This
shows that property b) is not valid in general. (A different
density measure $\mu$ and a set $A$ with $\mu(2A) \ne \frac12
\mu(A)$ was given by van Douwen \cite[Example 5.6, Case
2]{vandouwen}.)

The question (a) is closely related to van Douwen's question.
Clearly, if a set $A$ fulfills $\ud(A)<\mu(A)$ there is a set $B$
having asymptotic density $d(B)\in(\ud(A),\mu(A))$. Since
$d(B)>\ud(A)$, there exists $n_0$ such that $B(n)\geq A(n)$ for
$n>n_0$. Since changing only finitely many elements influences
neither asymptotic density nor density measure, any such pair of
sets yields a counterexample to the property (a).

\noindent\textbf{Acknowledgement:} The authors are greatly
indebted to an anonymous referee for suggesting significant
simplifications of several of the proofs presented in the paper
and, first of all, for suggesting Example \ref{exblum} instead of
the (more complicated) examples proposed originally by the
authors.


\end{document}